\documentclass[oneside,reqno,american]{amsart}
\usepackage[T1]{fontenc}
\usepackage[utf8]{inputenc}
\usepackage{babel}
\usepackage{prettyref}
\usepackage{float}
\usepackage{mathrsfs}
\usepackage{mathtools}
\usepackage{amstext}
\usepackage{amsthm}
\usepackage{amssymb}
\usepackage[pdfusetitle,
 bookmarks=true,bookmarksnumbered=false,bookmarksopen=false,
 breaklinks=false,pdfborder={0 0 1},backref=false,colorlinks=false]
 {hyperref}
\hypersetup{
 colorlinks=true,citecolor=blue,linkcolor=blue,linktocpage=true}

\makeatletter
\numberwithin{equation}{section}
\numberwithin{figure}{section}
\theoremstyle{plain}
\newtheorem{thm}{\protect\theoremname}[section]
\theoremstyle{remark}
\newtheorem{rem}[thm]{\protect\remarkname}
\theoremstyle{plain}
\newtheorem{cor}[thm]{\protect\corollaryname}
\theoremstyle{definition}
\newtheorem{defn}[thm]{\protect\definitionname}

\usepackage{prettyref}
\newrefformat{cor}{Corollary~\ref{#1}}
\newrefformat{subsec}{Section~\ref{#1}}
\newrefformat{lem}{Lemma~\ref{#1}}
\newrefformat{thm}{Theorem~\ref{#1}}
\newrefformat{sec}{Section~\ref{#1}}
\newrefformat{chap}{Chapter~\ref{#1}}
\newrefformat{prop}{Proposition~\ref{#1}}
\newrefformat{exa}{Example~\ref{#1}}
\newrefformat{tab}{Table~\ref{#1}}
\newrefformat{rem}{Remark~\ref{#1}}
\newrefformat{def}{Definition~\ref{#1}}
\newrefformat{fig}{Figure~\ref{#1}}
\newrefformat{claim}{Claim~\ref{#1}}

\usepackage{tikz}
\usetikzlibrary{arrows.meta, positioning}

\makeatother

\providecommand{\corollaryname}{Corollary}
\providecommand{\definitionname}{Definition}
\providecommand{\remarkname}{Remark}
\providecommand{\theoremname}{Theorem}

\begin{document}
\title[]{Gaussian processes in Non-commutative probability}
\author{Palle E.T. Jorgensen}
\address{(Palle E.T. Jorgensen) Department of Mathematics, The University of
Iowa, Iowa City, IA 52242-1419, U.S.A.}
\email{palle-jorgensen@uiowa.edu}
\author{James Tian}
\address{(James F. Tian) Mathematical Reviews, 416 4th Street Ann Arbor, MI
48103-4816, U.S.A.}
\email{jft@ams.org}
\begin{abstract}
Motivated by questions in quantum theory, we study Hilbert space valued
Gaussian processes, and operator-valued kernels, i.e., kernels taking
values in $B(H)$ ($=$ all bounded linear operators in a fixed Hilbert
space $H$). We begin with a systematic study of p.d. $B(H)$-valued
kernels and the associated of $H$-valued Gaussian processes, together
with their correlation and transfer operators. In our consideration
of $B(H)$-valued kernels, we drop the p.d. assumption. We show that
input-output models can be computed for systems of signed kernels
taking the precise form of realizability via associated transfer block
matrices (of operators analogous to the realization transforms in
systems theory), i.e., represented via $2\times2$ operator valued
block matrices. In the context of $B(H)$-valued kernels we present
new results on regression with $H$-valued Gaussian processes.
\end{abstract}

\maketitle

\section{Introduction}

The traditional framework for Gaussian processes and their covariance
kernels is the familiar case when the Gaussian processes under discussion
are scalar valued. And of course then the corresponding covariance
kernel will be positive definite (p.d.). Conversely, it is known that
every p.d. kernel arises as a covariance kernel from a (scalar valued)
Gaussian process. But for many questions in non-commutative analysis,
and in quantum theory, one must pass to the operator-valued case.
We make this precise for the case of operator valued kernels, i.e.,
kernels taking values in $B(H)$ ($=$ all bounded linear operators
in a fixed Hilbert space $H$). Our paper explores results and details
for this framework. In outline, the presentation below begins with
a systematic study of p.d. $B(H)$-valued kernels. And in \prettyref{sec:4},
we then present the corresponding results for the case of $H$-valued
Gaussian processes. In both sections, we make precise the links between
the two cases, i.e., contrasting scalar valued vs operator-valued.
Two main issues are resolved for the non-commutative (operator-valued)
setting: (i) for pairs of kernels, in Theorems \ref{thm:3-1} and
\prettyref{thm:3-4} we make precise the corresponding transfer operators,
and their realization; and (ii) in \prettyref{sec:4}, we make it
precise by relate it directly to the study of correlation for the
associated pairs of $H$-valued Gaussian processes. 

Our analysis of $B(H)$-valued kernels, and the corresponding Gaussian
processes from quantum probability (QP), allow us to expand the traditional
framework for Gaussian process-regression (GPR). We recall that GPR
is a powerful tool used in such areas as non-parametric regression
tasks in machine learning, and in statistics. Applications of the
results in our paper allow us to expand the classical setting of GPR
via use of our QP-Gaussian processes on observed data in making predictions
about unknown quantum states.

In earlier papers \cite{JT24-2,JT24-1,JT24-3} the authors introduced
a noncommutative framework for the study of quantum Gaussian processes.
The starting point in these results is the setting of operator values
positive definite kernels (p.d), i.e., $B(H)$-valued kernels taking
values in $B(H)$ for some Hilbert space $H$. Here in our present
discussion, we also consider $B(H)$ kernels, but we drop the p.d.
assumptions. Hence, we study pairs $K,L$ of $B(H)$-valued kernels,
with each kernel now arising as a difference of two p.d. kernels.

We demonstrate here that the study of such systems of signed kernels
is closely related to realizability for associated transfer matrices/operators
(from systems theory), taking the form of representations via $2\times2$
operator valued block matrices. For related papers on systems theory,
see e.g, \cite{MR133686,MR1465432,MR1873434,MR1957866,MR2397845,MR2543972,MR2992027,MR3647183,MR4557801,MR902603}.

\subsection*{Contents and main results}

In \prettyref{sec:2}, we present the basic tools for $B(H)$-valued
kernels which will be needed. \prettyref{sec:3} deals with systems
of $B(H)$-valued kernels, their properties, and their applications;
introducing: (i) transfer operators for pairs of kernels, \prettyref{thm:3-1};
(ii) corresponding input-output system realizations; (iii) ordering
of $B(H)$-valued kernels; and (iv) non-commutative Radon-Nikodym
derivatives, \prettyref{thm:3-4}. In \prettyref{sec:4}, we show
that the results (from Sections \ref{sec:2} and \ref{sec:3}) for
$B(H)$-valued kernels have counterparts for the corresponding $H$-valued
Gaussian processes. In particular, we show that pairs of $H$-valued
Gaussian processes admit the notions of conditioning, and joint distributions
(\prettyref{thm:d6}), i.e., a precise formulation of “jointly Gaussian”
for the $H$-valued case, and their properties (\prettyref{cor:4-7}).

\section{$B\left(H\right)$-valued positive definite kernels}\label{sec:2}

Our present paper deals with two closely related issues, (i) a number
of questions from operator theory, factorizations, and dilations;
and from (ii) non-commutative stochastic analysis. The first one addresses
questions around classes of Hilbert space dilations; and the second
(ii) deals with a general framework for Hilbert space valued Gaussian
processes. While, on the face of it, it is not obvious that the two
questions are closely related, but as we show, (i) and (ii) together
suggest the need for an extension of both the traditional framework
of positive definite (scalar valued) kernels and their associated
reproducing kernel Hilbert spaces (RKHSs); and of the framework for
non-commutative stochastic analysis. Hence our present starting point
will be a fixed separable Hilbert space $H$, and we then study the
notion of positive definite kernels $K$ which take values in the
algebra of all bounded operators on $H$, here denoted $B(H)$. For
background references on non-commutative probability, see e.g., \cite{MR4541766,MR4067990,MR4751140,MR4274591,MR4762721,MR4712050}
and the papers cited there.

A variety of Hilbert completions in our approach will now take an
explicit form, as opposed the guise of abstract equivalence classes.
In particular, every Hilbert space arising via a GNS (Gelfand-Naimark-Segal)
construction will now instead be a concrete RKHS. Hence, explicit
formulas, and no more abstract and inaccessible equivalence classes.

For the theory of reproducing kernels, the literature is extensive,
both pure and recent applications, we refer to the following, including
current, \cite{MR80878,MR31663,MR4302453,MR3526117,MR1305949,MR4690276,MR3700848}. 

\subsection*{Notations}

Throughout the paper, we continue with the physics convention: inner
products are linear in the second variable. 

For a p.d. kernel $K:S\times S\rightarrow\mathbb{C}$, we denote by
$H\left(K\right)$ the corresponding reproducing kernel Hilbert space.
It is the Hilbert completion of 
\[
span\left\{ K_{y}\left(\cdot\right):=K\left(\cdot,y\right)\mid y\in S\right\} 
\]
with respect to the inner product 
\[
\left\langle \sum_{i}c_{i}K\left(\cdot,x_{i}\right),\sum_{i}d_{j}K\left(\cdot,x_{j}\right)\right\rangle _{H\left(K\right)}:=\sum_{i}\sum_{j}\overline{c}_{i}d_{j}K\left(x_{i},x_{j}\right).
\]
The reproducing property is as follows: 
\[
\varphi\left(x\right)=\left\langle K\left(\cdot,x\right),\varphi\right\rangle _{H\left(K\right)},\forall x\in S,\:\forall\varphi\in H\left(K\right).
\]
For any orthonormal basis (ONB) $\left(\varphi_{i}\right)$ in $H\left(K\right)$,
one has 
\[
K\left(x,y\right)=\sum_{i}\varphi_{i}\left(x\right)\overline{\varphi_{i}\left(y\right)},\quad\forall x,y\in S.
\]

A $B\left(H\right)$-valued kernel $K:S\times S\rightarrow B\left(H\right)$
is positive definite (p.d.) if 
\[
\sum_{i=1}^{n}\left\langle a_{i},K\left(s_{i},s_{j}\right)a_{j}\right\rangle _{H}\geq0
\]
for all $\left(a_{i}\right)_{1}^{n}$ in $H$, $\left(s_{i}\right)_{1}^{n}$
in $S$, and all $n\in\mathbb{N}$.

\subsection*{Operator-valued kernels}

Below we introduce the general framework for operator valued positive
definite kernels. Our analysis is applied to several dilation constructions,
both in operator theory, and in stochastic processes. The latter will
be taken up in \prettyref{sec:4} again where we apply the results
here to build Hilbert space valued Gaussian processes, and examine
their properties. 

Let $S$ be a set, and let $K:S\times S\rightarrow B\left(H\right)$
be positive definite. Set $X=S\times H$, and define $\tilde{K}:X\times X\rightarrow\mathbb{C}$
by 
\[
\tilde{K}\left(\left(s,a\right),\left(t,b\right)\right)=\left\langle a,K\left(s,t\right)b\right\rangle _{H},
\]
for all $a,b\in H$, and all $s,t\in S$. Then $\tilde{K}$ is a scalar-valued
p.d. kernel on $X\times X$. Let $H(\tilde{K})$ be the corresponding
RKHS. 

It will follow from our subsequent results that this $H(\tilde{K})$
realization has properties that makes it universal (\prettyref{cor:B2}).
Moreover, one notes that in different but related p.d. settings, analogous
$H(\tilde{K})$-constructions serve as generators of dilation Hilbert
spaces, and representations. It has counterparts for such related
constructions as GNS, and Stinespring (complete positivity), see e.g.,
\cite{MR4114386,MR4581177,MR4482713,MR69403}.

The authors have chosen to include \prettyref{thm:B1}. Three reasons:
(i) this will allow us an opportunity to define and introduce the
basic notions, and terminology, which will be used throughout the
paper; (ii) for the benefit of non-specialists, it will make our paper
more self-contained. And, (iii) it will allow us to link the (more
familiar) classical theory for kernels with the new directions which
form the framework for the results presented later, Gaussian processes,
and the quantum framework.
\begin{thm}[\cite{JT24-1}]
\label{thm:B1}Consider the family of operators $\left\{ V_{K}\left(s\right)\right\} _{s\in S}$,
where $V_{K}\left(s\right):H\rightarrow H(\tilde{K})$, defined as
\[
V_{K}\left(s\right)a=\tilde{K}\left(\cdot,\left(s,a\right)\right):X\rightarrow\mathbb{C},\quad a\in H.
\]
Then the adjoint $V_{K}\left(s\right)^{*}:H(\tilde{K})\rightarrow H$
is determined by
\[
V_{K}\left(s\right)^{*}\tilde{K}\left(\cdot,\left(t,b\right)\right)=K\left(s,t\right)b
\]
for all $s,t\in S$, and $b\in H$. In particular, 
\[
V_{K}\left(s\right)V_{K}\left(s\right)^{*}\tilde{K}\left(\cdot,\left(t,b\right)\right)=\tilde{K}\left(\cdot,\left(s,K\left(s,t\right)b\right)\right)
\]
and 
\[
V_{K}\left(s\right)^{*}V_{K}\left(t\right)=K\left(s,t\right).
\]
\end{thm}

\begin{rem}
The theorem generalizes classical notions such as Stinespring's dilatation
of completely positive maps $\varphi$ on $C^{*}$-algebras $\mathfrak{A}$,
where $K\left(s,t\right)=\varphi\left(s^{*}t\right)$ for $s,t\in\mathfrak{A}$.
The RKHS $H(\tilde{K})$ plays the role of the dilation space, realized
as a space of functions on $S\times H$. 

Note that when $K\left(s,s\right)=I$, for all $s\in S$, which is
the case for unital CP maps, the operators $V_{K}\left(s\right)$
are isometric, since 
\[
\left\Vert V_{K}\left(s\right)a\right\Vert _{H(\tilde{K})}^{2}=\left\langle a,K\left(s,s\right)a\right\rangle _{H}=\left\langle a,a\right\rangle _{H}.
\]
Thus, $V_{K}\left(s\right)V_{K}\left(s\right)^{*}:H(\tilde{K})\rightarrow H(\tilde{K})$,
\[
V_{K}\left(s\right)V_{K}\left(s\right)^{*}\tilde{K}\left(\cdot,\left(t,b\right)\right)=\tilde{K}\left(\cdot,\left(s,K\left(s,t\right)b\right)\right)
\]
are projections in $H(\tilde{K})$. In particular, with an iteration,
\begin{eqnarray*}
 &  & V_{K}\left(s_{1}\right)V_{K}\left(s_{1}\right)^{*}\cdots V_{K}\left(s_{n}\right)V_{K}\left(s_{n}\right)^{*}\tilde{K}\left(\cdot,\left(t,b\right)\right)\\
 & = & \tilde{K}\left(\cdot,\left(s_{1},K\left(s_{1},s_{2}\right)\cdots K\left(s_{n-1},s_{n}\right)K\left(s_{n},t\right)b\right)\right).
\end{eqnarray*}
The importance of this is that iterated projections amounts to Kaczmarz
type algorithms and Parseval frames \cite{MR1881441}. 

\end{rem}

The above theorem provides a characterization for operator-valued
p.d. kernels: 
\begin{cor}[\cite{JT24-1}]
\label{cor:B2}Let $K:S\times S\rightarrow B\left(H\right)$. Then
the kernel $K$ is p.d. if and only if there exists a family of operators
$\left\{ V_{s}\right\} _{s\in S}$ from $H$ to another Hilbert space
$\mathscr{L}$, such that $K$ factors as 
\[
K\left(s,t\right)=V_{s}^{*}V_{t},\quad s,t\in S.
\]
If $K\left(s,s\right)=I$ for all $s\in S$, then $V_{s}:H\rightarrow\mathscr{L}$
are isometric. 

Furthermore, if $\mathscr{L}$ is minimal in the sense that
\[
\mathscr{L}=\overline{span}\left\{ V_{s}a:s\in S,\:a\in H\right\} ,
\]
then $\mathscr{L}\simeq H(\tilde{K}).$
\end{cor}

\section{Transfer functions }\label{sec:3}

In systems theory, there is a notion of realization dealing with state
space models. It asks for an implementation/realization of a given
input-output behavior. In more detail, there system identification
take the experimental data from a system and then output a realization.
Such techniques can utilize both input and output data (e.g. eigensystem
realization algorithm) or can only include the output data (e.g. frequency
domain decomposition), see e.g., \cite{MR1465432}. The realizations
of interest, allow representations via fractional linear transformations
by actions of $2\times2$ block-operator matrices, i.e., four operator
entries. The realizations we encounter below are in our present context
of pairs $K,L$ of $B(H)$-valued kernels. Each kernel in the pair
$K,L$ in turn will have a representation as the difference of two
p.d. $B(H)$-valued kernels. In the theorem below, \prettyref{thm:3-1},
we show that there is then an associated transfer operator function
with an explicit realization \eqref{eq:c10}. 

The work in \cite{MR4581177} by J. E. Pascoe and Ryan Tully-Doyle
concerns the relationships between completely positive (CP) maps on
$C^{*}$-algebras. In particular, it introduces the Agler class of
maps, providing transfer function type realizations through Stinespring
coefficients and Wittstock decompositions, which are motivated by
Nevanlinna-Pick interpolation and non-commutative function theory.
The current setting expands on these ideas by focusing on the explicit
construction of transfer operator functions for pairs of $B\left(H\right)$-valued
kernels, drawing parallels to the formulation by Pascoe and Tully-Doyle.
Especially, in the case when $K\left(s,t\right)=\varphi\left(s^{*}t\right)$
is CP map on a $C^{*}$-algebra, the factorization of $K$ in \prettyref{thm:B1}
reduces to $K\left(s,t\right)=V_{s}^{*}V_{t}=V^{*}\pi\left(s^{*}t\right)V$,
where $\left(\varphi,\pi,V,\mathscr{L}\right)$ is the Stinespring
dilation of $\varphi$, and one may choose $\mathscr{L}=H(\tilde{K})$
\cite{JT24-2}. 
\begin{thm}
\label{thm:3-1}Let $K,L:S\times S\rightarrow B\left(H\right)$. Assume
$K=K_{1}-K_{2}$, $L=L_{1}-L_{2}$, and 
\[
K=T^{*}LT
\]
where $K_{i},L_{i}$ are p.d., $i=1,2$, and $T\in B\left(H\right)$.
Let 
\begin{align*}
K_{i} & \left(s,t\right)=V_{K_{i}}\left(s\right)^{*}V_{K_{i}}\left(t\right)\\
L_{i} & \left(s,t\right)=V_{L_{i}}\left(s\right)^{*}V_{L_{i}}\left(t\right)
\end{align*}
be the factorization through the respective RKHSs (see \prettyref{thm:B1}),
i.e., 
\[
V_{K_{i}}\left(s\right):H\rightarrow H(\tilde{K}_{i}),\quad V_{L_{i}}\left(s\right):H\rightarrow H(\tilde{L}_{i}),\;s\in S.
\]
\begin{enumerate}
\item There is a partial isometry 
\begin{equation}
W=\begin{bmatrix}A & B\\
C & D
\end{bmatrix}\label{eq:a3}
\end{equation}
with initial space 
\begin{equation}
\overline{span}\left\{ \begin{bmatrix}\begin{array}{l}
V_{K_{2}}\left(s\right)\\
V_{L_{1}}\left(s\right)T
\end{array}\end{bmatrix}x:s\in S,\;x\in H\right\} \subset H(\tilde{K}_{2})\oplus H(\tilde{L}_{1})\label{eq:a4}
\end{equation}
and final space 
\begin{equation}
\overline{span}\left\{ \begin{bmatrix}\begin{array}{l}
V_{K_{1}}\left(s\right)\\
V_{L_{2}}\left(s\right)T
\end{array}\end{bmatrix}x:s\in S,\;x\in H\right\} \subset H(\tilde{K}_{1})\oplus H(\tilde{L}_{2})
\end{equation}
such that 
\begin{equation}
\begin{bmatrix}\begin{array}{l}
V_{K_{1}}\left(s\right)\\
V_{L_{2}}\left(s\right)T
\end{array}\end{bmatrix}=\begin{bmatrix}A & B\\
C & D
\end{bmatrix}\begin{bmatrix}\begin{array}{l}
V_{K_{2}}\left(s\right)\\
V_{L_{1}}\left(s\right)T
\end{array}\end{bmatrix}.\label{eq:a6}
\end{equation}
\item Furthermore, if 
\begin{equation}
\left(V_{L_{2}}-DV_{L_{1}}\right)^{-1}\label{eq:a7}
\end{equation}
 exists, then 
\begin{equation}
V_{K_{1}}\left(s\right)=T_{12}\left(s\right)V_{K_{2}}\left(s\right)\label{eq:c8}
\end{equation}
with the transfer function 
\begin{equation}
T_{12}\left(s\right)\coloneqq A+BV_{L_{1}}\left(s\right)\left(V_{L_{2}}\left(s\right)-DV_{L_{1}}\left(s\right)\right)^{-1}C\label{eq:c9}
\end{equation}
from $H(\tilde{K}_{2})$ to $H(\tilde{K}_{1})$. Consequently, 
\begin{equation}
K_{1}\left(s,t\right)=V_{K_{1}}^{*}\left(s\right)T_{12}\left(t\right)V_{K_{2}}\left(t\right).\label{eq:c10}
\end{equation}
\end{enumerate}
\end{thm}

\begin{proof}
Let $K_{i}$ and $L_{i}$ be as stated. For all $s,t\in S$, and all
$x,y\in H$, the following are equivalent: 
\begin{gather*}
K=T^{*}LT\\
\Updownarrow\\
V_{K_{1}}\left(s\right)^{*}V_{K_{1}}\left(t\right)-V_{K_{2}}\left(s\right)^{*}V_{K_{2}}\left(t\right)=T^{*}V_{L_{1}}\left(s\right)^{*}V_{L_{1}}\left(t\right)T-T^{*}V_{L_{2}}\left(s\right)^{*}V_{L_{2}}\left(t\right)T\\
\Updownarrow\\
V_{K_{1}}\left(s\right)^{*}V_{K_{1}}\left(t\right)+T^{*}V_{L_{2}}\left(s\right)^{*}V_{L_{2}}\left(t\right)T=V_{K_{2}}\left(s\right)^{*}V_{K_{2}}\left(t\right)+T^{*}V_{L_{1}}\left(s\right)^{*}V_{L_{1}}\left(t\right)T\\
\Updownarrow\\
\left\langle \begin{bmatrix}\begin{array}{l}
V_{K_{1}}\left(s\right)\\
V_{L_{2}}\left(s\right)T
\end{array}\end{bmatrix}x,\begin{bmatrix}\begin{array}{l}
V_{K_{1}}\left(t\right)\\
V_{L_{2}}\left(t\right)T
\end{array}\end{bmatrix}y\right\rangle =\left\langle \begin{bmatrix}\begin{array}{l}
V_{K_{2}}\left(s\right)\\
V_{L_{1}}\left(s\right)T
\end{array}\end{bmatrix}x,\begin{bmatrix}\begin{array}{l}
V_{K_{2}}\left(t\right)\\
V_{L_{1}}\left(t\right)T
\end{array}\end{bmatrix}y\right\rangle 
\end{gather*}
It follows that 
\[
\begin{bmatrix}\begin{array}{l}
V_{K_{2}}\left(s\right)\\
V_{L_{1}}\left(s\right)T
\end{array}\end{bmatrix}x\longmapsto\begin{bmatrix}\begin{array}{l}
V_{K_{1}}\left(s\right)\\
V_{L_{2}}\left(s\right)T
\end{array}\end{bmatrix}x
\]
is isometric. 

Let $W\coloneqq\begin{bmatrix}A & B\\
C & D
\end{bmatrix}$ be the corresponding partial isometry, i.e., \eqref{eq:a3}--\eqref{eq:a6}.
Thus, 
\begin{align}
V_{K_{1}} & =AV_{K_{2}}+BV_{L_{1}}T\label{eq:1}\\
V_{L_{2}}T & =CV_{K_{2}}+DV_{L_{1}}T\label{eq:2}
\end{align}

From \eqref{eq:2}, we get 
\[
T=\left(V_{L_{2}}-DV_{L_{1}}\right)^{-1}CV_{K_{2}}
\]
assuming the inverse exists. Substitute this into \eqref{eq:1}: 
\begin{align*}
V_{K_{1}} & =AV_{K_{2}}+BV_{L_{1}}T\\
 & =AV_{K_{2}}+BV_{L_{1}}\left(V_{L_{2}}-DV_{L_{1}}\right)^{-1}CV_{K_{2}}\\
 & =\left(A+BV_{L_{1}}\left(V_{L_{2}}-DV_{L_{1}}\right)^{-1}C\right)V_{K_{2}}
\end{align*}
which yield \eqref{eq:c8}--\eqref{eq:c9}, and \eqref{eq:c10} follows
since $K_{1}\left(s,t\right)=V_{K_{1}}\left(s\right)^{*}V_{K_{1}}\left(t\right)$. 
\end{proof}
The relationship between $K$ and $L$ in \prettyref{thm:3-1} can
be written as 
\begin{equation}
K_{1}-T^{*}L_{1}T=K_{2}-T^{*}L_{2}T.\label{eq:c11}
\end{equation}
This yields the following equivalence relation:
\begin{defn}
Let $Pos\left(S\right)$ be the set of $B\left(H\right)$-valued p.d.
kernels on $S\times S$. Fix an operator $T\in B\left(H\right)$,
and let 
\[
K_{1}\sim_{T}K_{2}\Longleftrightarrow\exists L_{1},L_{2}\in Pos\left(S\right)\:\text{s.t.}\:(\ref{eq:c11})\;\text{holds.}
\]
Denote by $\left[K\right]_{T}=\left\{ K'\in Pos\left(S\right):K'\sim_{T}K\right\} $,
and $Pos\left(S\right)/\sim_{T}$ the set of all equivalence classes. 
\end{defn}

\begin{cor}
Given $T\in B\left(H\right)$, there is a transitive action within
the set of RKHSs associated with the equivalence class of any $K$
under $T$. Specifically, for any $K\in Pos\left(S\right)$, the set
\[
\left\{ H(\tilde{K}'):K'\in[K]_{T}\right\} 
\]
exhibits a transitive action via the transfer function \eqref{eq:c8},
provided the invertibility condition \eqref{eq:a7} is satisfied. 
\end{cor}

\begin{proof}
Suppose $K_{1}\sim_{T}K_{2}$. By \prettyref{thm:B1}, 
\begin{align*}
H(\tilde{K}_{1}) & =\overline{span}\left\{ V_{K_{1}}\left(s\right)a:s\in S,\;a\in H\right\} ,\\
H(\tilde{K}_{2}) & =\overline{span}\left\{ V_{K_{2}}\left(s\right)a:s\in S,\;a\in H\right\} .
\end{align*}
The transfer function from \eqref{eq:c8}, denoted by $T_{12}$, satisfies
that 
\[
V_{K_{1}}\left(s\right)a=T_{21}\left(s\right)V_{K_{2}}\left(s\right)a
\]
for all $s\in S$, and $a\in H$. Thus, 
\[
H(\tilde{K}_{1})=T_{12}(H(\tilde{K}_{2})).
\]
\end{proof}
Next, we identify a non-commutative Radon-Nikodym derivative from
the system $\left(K_{1},K_{2},L_{1},L_{2}\right)$ under the assumption
that $K_{1}\leq K_{2}$. 

Here, for $K$ and $L$ p.d. kernels, $L\leq K\Longleftrightarrow K-L$
is p.d., i.e., 
\[
\sum_{i,j=1}^{n}\left\langle h_{i},L\left(s_{i},s_{j}\right)h_{j}\right\rangle _{H}\leq\sum_{i,j=1}^{n}\left\langle h_{i},K\left(s_{i},s_{j}\right)h_{j}\right\rangle _{H}
\]
for all $\left(s_{i}\right)_{1}^{n}$ in $S$, $\left(h_{i}\right)_{1}^{n}$
in $\mathscr{H}$, and all $n\in\mathbb{N}$.
\begin{thm}
\label{thm:3-4}Suppose $K_{1}\sim_{T}K_{2}$, i.e., $K_{1}-T^{*}L_{1}T=K_{2}-T^{*}L_{2}T$
with $L_{1},L_{2}\in Pos\left(S\right)$. Assume the invertibility
condition \eqref{eq:a7} holds, and $K_{1}\leq K_{2}$. Then 
\[
\sqrt{\frac{dK_{1}}{dK_{2}}}=T_{12}
\]
where $T_{12}$ is the transfer function in \eqref{eq:c9}. 

Especially, $T_{12}:H(\tilde{K}_{2})\rightarrow H(\tilde{K}_{1})\subset H(\tilde{K}_{2})$
is a positive self-adjoint operator in $H(\tilde{K}_{2})$. 
\end{thm}

\begin{proof}
We will use the following result \cite{JT24-3}: For p.d. kernels
$K,L:S\times S\rightarrow B\left(H\right)$, it holds that 
\[
L\leq K\Longleftrightarrow L\left(s,t\right)=V_{K}\left(s\right)^{*}\Phi V_{K}\left(t\right)
\]
where $\Phi$ is a unique self-adjoint operator in $H(\tilde{K})$
with $0\leq\Phi\leq I$. 

The operator $\Phi$ will be referred to as the Radon-Nikodym derivative
of $L$ with respect to $K$, denoted by $\Phi=dL/dK$. In this case,
we have 
\[
V_{L}\left(s\right)=\Phi^{1/2}V_{K}\left(s\right)
\]
and $H(\tilde{K})\subset H(\tilde{L})$.

Now consider $K_{1}-T^{*}L_{1}T=K_{2}-T^{*}L_{2}T$ as in the statement
of the theorem. By assuming $K_{1}\leq K_{2}$, we get 
\[
V_{K_{1}}\left(s\right)=\Phi^{1/2}V_{K_{2}}\left(s\right).
\]
But it follows from \prettyref{thm:3-1} that 
\[
V_{K_{1}}\left(s\right)=T_{12}\left(s\right)V_{K_{2}}\left(s\right).
\]
Thus, the uniqueness of the Radon-Nikodym derivative, along with the
density of 
\[
span\left\{ V_{K_{2}}\left(s\right)a:s\in S,\:a\in H\right\} 
\]
in $H(\tilde{K}_{2})$, shows that $\Phi^{1/2}=T_{12}$. 
\end{proof}
Our results here are part of a non-commutative approach to stochastic
calculus. It will be made explicit in \prettyref{sec:4} below where
we present an explicit correspondence between (i) $B(H)$-valued positive
definite kernels $K$ on the one and associated (ii) $H$-valued Gaussian
processes $W$, on the other, see especially Definitions \ref{def:d1}
and \ref{def:d2}; combined with \prettyref{thm:B4}.

The correspondence is explicit in both directions, from $K$ to $W$
in \eqref{eq:d8}, and from $W$ to $K$ in \prettyref{cor:d4}. Our
transfer function realization from \prettyref{thm:3-1} is illustrated
in \prettyref{fig:c1} below.

\begin{figure}[H]
\begin{tikzpicture}[node distance=1.2cm, auto]

  \node (matrix1) at (0, 0) {
    $\begin{pmatrix}
      W_s^{(L_1)} \\
      W_t^{(L_2)}
    \end{pmatrix}$
  };

  \node (matrix2) [right=of matrix1, draw, thick, minimum width=1cm, minimum height=1.3cm, align=center] {
    $\begin{pmatrix}
      A & B \\
      C & D
    \end{pmatrix}$
  };

  \node (matrix3) [right=of matrix2] {
    $\begin{pmatrix}
      W_s^{(K_1)} \\
      W_t^{(K_2)}
    \end{pmatrix}$
  };

  \draw[->] ([yshift=2mm, xshift=-2mm]matrix1.east) -- ([yshift=2mm, xshift=0mm]matrix2.west);
  \draw[->] ([yshift=-2mm, xshift=-2mm]matrix1.east) -- ([yshift=-2mm, xshift=0mm]matrix2.west);

  \draw[->] ([yshift=2mm, xshift=0mm]matrix2.east) -- ([yshift=2mm, xshift=2mm]matrix3.west);
  \draw[->] ([yshift=-2mm, xshift=0mm]matrix2.east) -- ([yshift=-2mm, xshift=2mm]matrix3.west);

\end{tikzpicture}

\caption{ Input/output model for the $H$-valued Gaussian processes $\left(W_{s}\right)$
corresponding to the four p.d. kernels from \prettyref{thm:3-1}.}\label{fig:c1}
\end{figure}
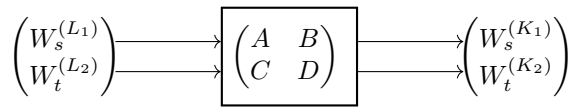

\section{$H$-valued Gaussian processes}\label{sec:4}

In probability and stochastic, we think of Gaussian processes as follows:
Specify a probability space, and a system of random variables, indexed
by a set, say $S$, so $\left\{ W_{s},s\in S\right\} $. We say that
this is a Gaussian process, if the system $\left\{ W_{s},s\in S\right\} $
is jointly Gaussian. The corresponding covariance kernel will then
be positive definite (p.d.), and by a standard construction, one can
show that every positive definite kernel arises this way. Note that
from first principles, every p.d. kernel on $S\times S$ arises this
way from a choice of Gaussian process $\left\{ W_{s},s\in S\right\} $.
In particular, it follows that Gaussian processes are determined by
their first two moments; where the p.d. kernel in question is interpreted
as second moments. The mean of the process is “first moments.” Suppose
$\left\{ W_{h},h\in H\right\} $ is indexed by a Hilbert space $H$.
If the inner product of $H$ coincides with the second moments, i.e.,
the covariance kernel, we talk about a Gaussian Hilbert space. See
e.g., \cite{MR4641110}.

Since Ito's early work, the stochastic analysis of Gaussian processes
has often relied on Hilbert space techniques, as seen in \cite{MR3402823,MR4302453,MR45307}
and related literature. Here we want to call attention to the following
distinction: (i) scalar-valued Gaussian processes, potentially indexed
by a Hilbert space via an Ito-isometry, and (ii) Hilbert space-valued
Gaussian processes, see e.g., \cite{MR4641110,MR4414825,MR4101087,MR4073554,MR3940383}. 

We focus on the latter (ii), as it offers a more versatile approach
to model complex covariance structures in large datasets. This construction,
detailed below, is based on the results from \prettyref{sec:2}, utilizing
both $B(H)$-valued positive definite kernels, and their induced scalar
valued dilations, see \prettyref{thm:B1} and Corollary \ref{cor:B2}.

A key feature of Gaussian processes is that they are determined by
their first two moments, also when interpreted in a more general metric
framework. See e.g., \cite{MR4728479,MR4260603,MR3447225}.

We now turn to the details for our quantum kernel method. We believe
that it is appealing because it has parallels to the well established
toolbox of classical kernel theory for feature selection. And our
$B(H)$-valued kernels allow us to incorporate quantum features.

Let $S$ be a set, and $H$ a Hilbert space. We shall introduce choices
of probability spaces $(\Omega,\mathcal{F},\mathbb{P})$, i.e., sample
space $\Omega$, sigma-algebra $\mathcal{F}$, and probability measure
$\mathbb{P}$ defined on $\mathcal{F}$. Then a $H$-valued Gaussian
process $W$ makes reference to a choice of probability space $(\Omega,\mathcal{F},\mathbb{P})$.
We say that $W$ is a H-valued random process indexed by $S$ if the
following holds: $W$ is specified as a random function $W$ (written
$W_{s}$) from $S$ into $H$ such that for all $a$ in $H$, and
$s$ in $S$, the corresponding system $\left\langle a,W_{s}\right\rangle =\left\langle a,W_{s}\left(\cdot\right)\right\rangle $
as scalar random variables defined on $\Omega$, has a jointly Gaussian
distribution, i.e., is a Gaussian process. (In fact, this is the scalar
valued Gaussian process defined from the kernel $\tilde{K}$ introduced
in \prettyref{thm:B1} where $K$ is a $B(H)$-valued p.d. kernel.
We say that $K$ is the $B(H)$ valued covariance kernel for $W$.)
More precisely: 
\begin{defn}
\label{def:d1} $\left\{ W_{s}\right\} _{s\in S}$ is an $H$-valued
Gaussian process if, for all $a,b\in H$, and all $s,t\in S$, 
\[
\mathbb{E}\left[\left\langle a,W_{s}\right\rangle _{H}\left\langle W_{t},b\right\rangle _{H}\right]=\left\langle a,K\left(s,t\right)b\right\rangle _{H}
\]
where $K$ is a $B\left(H\right)$-valued p.d. kernel, and 
\[
\left\langle W_{s},a\right\rangle _{H}\sim N\left(0,\left\langle a,K\left(s,s\right)a\right\rangle _{H}\right),
\]
i.e., mean zero Gaussian, with variance $\left\langle a,K\left(s,s\right)a\right\rangle _{H}$.
\end{defn}

Fix $S\times S\xrightarrow{\;K\;}B\left(H\right)$, and define $\tilde{K}:\left(S\times H\right)\times\left(S\times H\right)\rightarrow\mathbb{C}$,
given by 
\[
\tilde{K}\left(\left(s,a\right),\left(t,b\right)\right)=\left\langle a,K\left(s,t\right)b\right\rangle _{H},\quad\forall s,t\in S,\ \forall a,b\in H,
\]
defined on $S\times H$. Let $H(\tilde{K})$ be the corresponding
RKHS (consisting of scalar-valued functions on $S\times H$). 

In what follows, we assume all $H(\tilde{K})$ induced by $K$ are
separable. 
\begin{defn}
\label{def:d2}Consider the $H$-valued Gaussian process $W_{s}:\Omega\rightarrow H$,
\begin{equation}
W_{t}=\sum_{i}\left(V_{t}^{*}\varphi_{i}\right)Z_{i},\label{eq:d8}
\end{equation}
where $\left(\varphi_{i}\right)$ is an ONB in $H(\tilde{K})$. Following
standard conventions, here $\left\{ Z_{i}\right\} $ refers to a choice
of an independent identically distributed (i.i.d.) system of standard
scalar Gaussian $N(0,1)$ random variables, and with an index matching
the choice of ONB.
\end{defn}

\begin{thm}[\cite{JT24-1,JT24-3}]
\label{thm:B4}We have 
\[
\mathbb{E}\left[\left\langle a,W_{s}\right\rangle _{H}\left\langle W_{t},b\right\rangle _{H}\right]=\left\langle a,K\left(s,t\right)b\right\rangle _{H}.
\]
\end{thm}

\prettyref{thm:B4} gives another characterization for operator-valued
p.d. kernels, in terms of direct integral decomposition:

\begin{cor}
\label{cor:d4}Let $K:S\times S\rightarrow B\left(H\right)$. Then
the kernel $K$ is p.d. if and only if it decomposes as
\[
K\left(s,t\right)=\int_{\Omega}\left|W_{s}\left\rangle \right\langle W_{t}\right|d\mathbb{P},
\]
where $\left\{ W_{t}\right\} _{t\in S}$ is the $H$-valued Gaussian
process from above. (In the expression under the integral we use Dirac’s
ket-bra notation for rank-one operators. Hence, for every $s,t$,
the operator $K(s,t)$ is the expectation of a random field of rank-one
operators.)
\end{cor}

\subsection*{Conditioning, joint distributions}

Consider a system of positive definite kernels $\left(K_{1},K_{2},L_{1},L_{2}\right)$
defined on $S\times S$, such that 
\begin{equation}
K_{1}-TL_{1}^{-1}T^{*}=K_{2}-TL_{2}^{-1}T^{*}\label{eq:d4}
\end{equation}
where $T:S\times S\rightarrow B\left(H\right)$, and $L_{1}^{-1},L_{2}^{-1}\in Pos\left(S\right)=$
the set of all p.d. kernels on $S\times S$. Note $T$ is assumed
to be a $B\left(H\right)$-valued function, instead of a fixed operator
in $B\left(H\right)$. 
\begin{rem}
An example where $L,L^{-1}\in Pos\left(S\right)$ is the following
completely positive map
\[
L\left(s,t\right)=I_{n}-\varphi_{h}\left(s^{*}t\right)
\]
for $s,t\in M_{n}\left(\mathbb{C}\right)$, where 
\[
\varphi_{h}\left(t\right)=h^{*}th
\]
for a fixed, strict contraction $h\in M_{n}\left(\mathbb{C}\right)$.
Here, 
\[
L^{-1}=\sum_{n=0}^{\infty}\varphi_{h}^{n}\in Pos\left(S\right).
\]
\end{rem}

\begin{thm}
\label{thm:d6}Let $K,L\in Pos\left(S\right)$, such that $K-TL^{-1}T^{*}\in Pos\left(S\right)$
for some $T:S\times S\rightarrow B\left(H\right)$. Then there exists
a joint $H$-valued Gaussian process 
\begin{equation}
\left\{ \left(W_{K}\left(s\right),W_{L}\left(s\right)\right)\right\} _{s\in S}\sim N\left(0,\begin{bmatrix}K & T\\
T^{*} & L
\end{bmatrix}\right).\label{eq:d-3}
\end{equation}
In particular, the covariance of $W_{K}$ and $W_{L}$ is specified
by 
\[
\mathbb{E}\left[\left\langle a,W_{K}\left(s\right)\right\rangle _{H}\left\langle W_{L}\left(t\right),b\right\rangle _{H}\right]=\left\langle a,T\left(s,t\right)b\right\rangle _{H}
\]
for all $s,t\in S$, and all $a,b\in H$. Equivalently, 
\[
T\left(s,t\right)=\int_{\Omega}\left|W_{K}\left(s\right)\left\rangle \right\langle W_{L}\left(t\right)\right|d\mathbb{P}.
\]

Moreover, 
\begin{equation}
W_{K\mid L}\coloneqq W_{K}\mid W_{L}\sim N\left(TL^{-1}W_{L},K-TL^{-1}T^{*}\right)\label{eq:d-4}
\end{equation}
in the sense that 
\[
\mathbb{E}\left[\left\langle W_{K\mid L}\left(s\right),a\right\rangle _{H}\right]=T^ {}\left(s,s\right)L^{-1}\left(s,s\right)\left\langle W_{L}\left(s\right),a\right\rangle _{H}
\]
and 
\[
Cov\left[\left\langle a,W_{K\mid L}\left(s\right)\right\rangle _{H}\left\langle W_{K\mid L}\left(t\right),b\right\rangle _{H}\right]=\left\langle a,\left(K-TL^{-1}T^{*}\right)\left(s,t\right)b\right\rangle _{H}.
\]
\end{thm}

\begin{proof}
We need to verify that 
\[
M\left(s,t\right)=\begin{bmatrix}K\left(s,t\right) & T\left(s,t\right)\\
T^{*}\left(s,t\right) & L\left(s,t\right)
\end{bmatrix}\in B\left(H\oplus H\right),\quad\left(s,t\right)\in S\times S
\]
is positive definite, or equivalently, 
\[
\sum_{i,j=1}^{n}\left\langle \begin{bmatrix}a_{i}\\
b_{i}
\end{bmatrix},M\left(s_{i},s_{j}\right)\begin{bmatrix}a_{j}\\
b_{j}
\end{bmatrix}\right\rangle \geq0
\]
for all finite sets of vectors $a_{i},b_{i}$ in $H$, and $s_{i}$
in $S$. This is true, since the Schur complement of $L$ in $M$
is assumed p.d., i.e., 
\[
M/L=K-TL^{-1}T^{*}\in Pos\left(S\right).
\]

Let $M=V_{M}\left(s\right)^{*}V_{M}\left(t\right)$ be the corresponding
factorization of $M$, via the RKHS $H(\tilde{M})$, consisting of
functions on $S\times\left(H\oplus H\right)$. Let $\left\{ \varphi_{i}\right\} $
be an ONB in $H(\tilde{M})$, and define $W_{K}$ and $W_{L}$ by
\begin{align*}
\left\langle W_{K}\left(s\right),a\right\rangle _{H} & =\sum_{i}\left\langle V_{M}^{*}\left(s\right)\varphi_{i},\begin{bmatrix}a\\
0
\end{bmatrix}\right\rangle _{H\oplus H}Z_{i}\\
\left\langle W_{L}\left(t\right),b\right\rangle _{H} & =\sum_{i}\left\langle V_{M}^{*}\left(t\right)\varphi_{i},\begin{bmatrix}0\\
b
\end{bmatrix}\right\rangle _{H\oplus H}Z_{i}
\end{align*}
for all $a,b\in H$, and all $s,t\in S$. Note that
\begin{eqnarray*}
 &  & \mathbb{E}\left[\left\langle a,W_{K}\left(s\right)\right\rangle _{H}\left\langle W_{L}\left(t\right),b\right\rangle _{H}\right]\\
 & = & \sum_{i,j}\left\langle \begin{bmatrix}a\\
0
\end{bmatrix},V_{M}^{*}\left(s\right)\varphi_{i}\right\rangle _{H\oplus H}\left\langle V_{M}^{*}\left(t\right)\varphi_{j},\begin{bmatrix}0\\
b
\end{bmatrix}\right\rangle _{H\oplus H}\mathbb{E}\left[Z_{i}Z_{j}\right]\\
 & = & \sum_{i}\left\langle \begin{bmatrix}a\\
0
\end{bmatrix},V_{M}^{*}\left(s\right)\varphi_{i}\right\rangle _{H\oplus H}\left\langle V_{M}^{*}\left(t\right)\varphi_{i},\begin{bmatrix}0\\
b
\end{bmatrix}\right\rangle _{H\oplus H}\\
 & = & \sum_{i}\left\langle V_{M}\left(s\right)\begin{bmatrix}a\\
0
\end{bmatrix},\varphi_{i}\right\rangle _{H(\tilde{M})}\left\langle \varphi_{i},V_{M}\left(t\right)\begin{bmatrix}0\\
b
\end{bmatrix}\right\rangle _{H(\tilde{M})}\\
 & = & \left\langle V_{M}\left(s\right)\begin{bmatrix}a\\
0
\end{bmatrix},V_{M}\left(t\right)\begin{bmatrix}0\\
b
\end{bmatrix}\right\rangle _{H(\tilde{M})}\\
 & = & \left\langle \begin{bmatrix}a\\
0
\end{bmatrix},V_{M}\left(s\right)^{*}V_{M}\left(t\right)\begin{bmatrix}0\\
b
\end{bmatrix}\right\rangle _{H\oplus H}\\
 & = & \left\langle \begin{bmatrix}a\\
0
\end{bmatrix},M\left(s,t\right)\begin{bmatrix}0\\
b
\end{bmatrix}\right\rangle _{H\oplus H}=\left\langle a,T\left(s,t\right)b\right\rangle _{H}.
\end{eqnarray*}
\end{proof}
\begin{defn}
\label{def:d-7}Given \eqref{eq:d-3}, we say $\left\{ W_{K}\left(s\right)\right\} _{s\in S}$
and $\left\{ W_{L}\left(s\right)\right\} _{s\in S}$ are (statistically)
independent if $T\equiv0$. 
\end{defn}

\begin{cor}
\label{cor:4-7}For a system $\left(K_{1},K_{2},L_{1},L_{2}\right)$,
we have 
\begin{gather*}
K_{1}-TL_{1}^{-1}T^{*}=K_{2}-TL_{2}^{-1}T^{*}\in Pos\left(S\right)\\
\Updownarrow\\
Cov\left(W_{K_{1}}\mid W_{L_{1}}\right)=Cov\left(W_{K_{2}}\mid W_{L_{2}}\right)
\end{gather*}
\end{cor}

\begin{proof}
By setting 
\[
M_{1}\coloneqq\begin{bmatrix}K_{1} & T\\
T^{*} & L_{1}
\end{bmatrix},\quad M_{2}\coloneqq\begin{bmatrix}K_{2} & T\\
T^{*} & L_{2}
\end{bmatrix}
\]
then \eqref{eq:d4} is equivalent to 
\begin{align*}
K_{1}-K_{2} & =T\left(L_{1}^{-1}-L_{2}^{-1}\right)T^{*}\\
 & \Updownarrow\\
M_{1}/L_{1} & =M_{1}/L_{2}\;\left(\text{Schur complements}\right)
\end{align*}
\end{proof}

\subsection*{Gaussian process regression}

Gaussian process regression (GPR) is a powerful non-parametric approach
used for regression tasks in machine learning and statistics. It uses
the properties of Gaussian processes to make predictions about unknown
functions based on observed data. GPR has a strong theoretical foundation
and offers flexibility in modeling complex data structures. It is
widely known that GPR and kernel ridge regression are equivalent. 

Our work in previous sections extends the classical setting of GPR
and ridge regression to the more general context of operator-valued
kernels: 

Let $K,L:S\times S\rightarrow B\left(H\right)$ be positive definite,
with the associated $H$-valued Gaussian processes $W_{K}$ and $W_{L}$.

Let $T\equiv0$ in \eqref{eq:d-3}, and consider 
\[
\left\{ \left(W_{K}\left(s\right),W_{L}\left(s\right)\right)\right\} _{s\in S}\sim N\left(0,\begin{bmatrix}K & 0\\
0 & L
\end{bmatrix}\right)
\]
That is, $W_{K}$ and $W_{L}$ are independent (\prettyref{def:d-7}).
Set 
\[
Y=W_{K}+W_{L}.
\]
It follows that 
\[
\left\{ W_{K}\left(s\right),W_{Y}\left(s\right)\right\} _{s\in S}\sim N\left(0,\begin{bmatrix}K & K\\
K & K+L
\end{bmatrix}\right).
\]
Therefore, by \prettyref{thm:d6} (see \eqref{eq:d-4}), we have the
conditional expectation: 
\begin{equation}
\mathbb{E}\left[W_{K}\mid W_{Y}\right]=K\left(K+L\right)^{-1}W_{Y}.\label{eq:d-5}
\end{equation}

This is a non-commutative version of the the standard GPR. We illustrate
below that this version of GPR is equivalent to kernel ridge regression,
in the setting of $B\left(H\right)$-valued kernels, and the induced
RKHS $H(\tilde{K})$. 

Fix a finite sample $\left\{ \left(s_{i},a_{i}\right)\right\} _{i=1}^{m}$
in $S\times H$, and $\left(y_{i}\right)\in\mathbb{C}^{m}$. Define
the following $m\times m$ matrices 
\begin{align*}
[\tilde{L}] & =[\tilde{L}\left(\left(s_{i},a_{i}\right),\left(s_{j},a_{j}\right)\right)],\\{}
[\tilde{K}] & =[\tilde{K}\left(\left(s_{i},a_{i}\right),\left(s_{j},a_{j}\right)\right)].
\end{align*}
Set $\Phi:H(\tilde{K})\rightarrow\mathbb{C}^{m}$ by 
\[
\Phi\left(f\right)=\left(f\left(s_{i},a_{i}\right)\right)\in\mathbb{C}^{m}.
\]
(Recall that elements in $H(\tilde{K})$ are functions $f\left(s,a\right)$
defined on $S\times H$.)

The adjoint $\Phi^{*}:\mathbb{C}^{m}\rightarrow H(\tilde{K})$ is
then given by 
\[
\Phi^{*}\left(\left(c_{i}\right)\right)=\sum_{i=1}^{m}c_{i}\tilde{K}\left(\cdot,\left(s_{i},a_{i}\right)\right),\quad\left(c_{i}\right)\in\mathbb{C}^{m}.
\]
Thus, 
\[
\Phi\Phi^{*}=[\tilde{K}]\in M_{m}\left(\mathbb{C}\right).
\]
Consider the following optimization problem: 
\begin{equation}
\inf_{f\in H(\tilde{K})}\left\{ \left\langle f\left(\mathbf{s},\mathbf{a}\right)-\mathbf{y},[\tilde{L}]^{-1}\left(f\left(\mathbf{s},\mathbf{a}\right)-\mathbf{y}\right)\right\rangle _{\mathbb{C}^{m}}+\left\Vert f\right\Vert _{H(\tilde{K})}^{2}\right\} .\label{eq:d-6}
\end{equation}

\begin{cor}
The unique solution to \eqref{eq:d-6} is given by 
\begin{equation}
f^{*}=\Phi^{*}\left([\tilde{L}]+[\tilde{K}]\right)^{-1}\mathbf{y},
\end{equation}
and so 
\begin{equation}
\Phi f^{*}=[\tilde{K}]\left([\tilde{L}]+[\tilde{K}]\right)^{-1}\mathbf{y}.\label{eq:4-8}
\end{equation}
\end{cor}

\begin{rem}
From a Bayesian perspective, ridge regression can be interpreted as
a maximum a posteriori (MAP) estimation with a Gaussian prior. This
provides a probabilistic interpretation of the regularization term
in ridge regression. Specifically, the conditional expectation in
\eqref{eq:d-5} represents the Bayesian update of the Gaussian process
given the observed data. Thus, the optimization problem in \eqref{eq:d-6}
amounts to finding the MAP estimate $\Phi f$ with the function $f$
in the RKHS $H(\tilde{K})$. 
\end{rem}

The ability to handle $H$-valued data points $\left\{ \left(\left(s_{i},a_{i}\right),y_{i}\right)\right\} $
allows for more sophisticated modeling of complex relationships. For
example, in spatial statistics, $s_{i}$ could represent spatial locations,
$a_{i}$ can be multivariate measurements taken at these locations,
and $y_{i}$ are the observed outcomes. In this context, our framework
captures dependencies and interactions between different locations
and measurements, leading to more accurate predictions and a better
understanding of underlying patterns.

Additionally, in machine learning applications, this approach can
be used for multi-task learning, where each $a_{i}$ represents different
tasks or conditions, and $y_{i}$ are the corresponding observations.
By means of operator-valued kernels, the model can effectively share
information across tasks, improving performance in situations with
limited data for individual tasks.

In the existing literature, GPR with operator-valued kernels typically
uses a vector-valued RKHS, consisting of functions taking values in
a Hilbert space $H$. Our approach, however, uses $H(\tilde{K})$
with an induced scalar-valued kernel $\tilde{K}$. While the two are
equivalent, our scalar-valued approach offers several practical advantages.
For instance, it simplifies the mathematical framework, making it
easier to apply existing scalar-valued techniques and algorithms.
This reduction in complexity enhances computational efficiency, allows
for faster processing and the handling of larger datasets. It also
offers a clearer interpretation of model parameters and results, providing
better insights into data relationships and dependencies.

\bibliographystyle{amsalpha}
\nocite{*}
\bibliography{bib}

\providecommand{\bysame}{\leavevmode\hbox to3em{\hrulefill}\thinspace}
\providecommand{\MR}{\relax\ifhmode\unskip\space\fi MR }
\providecommand{\MRhref}[2]{%
  \href{http://www.ams.org/mathscinet-getitem?mr=#1}{#2}
}
\providecommand{\href}[2]{#2}
\begin{thebibliography}{SNFBK10}

\bibitem[AADL03]{MR1957866}
D.~Alpay, T.~Ya. Azizov, A.~Dijksma, and H.~Langer, \emph{The {S}chur algorithm
  for generalized {S}chur functions. {II}. {J}ordan chains and transformations
  of characteristic functions}, Monatsh. Math. \textbf{138} (2003), no.~1,
  1--29. \MR{1957866}

\bibitem[AB97]{MR1473250}
D.~Alpay and V.~Bolotnikov, \emph{On tangential interpolation in reproducing
  kernel {H}ilbert modules and applications}, Topics in interpolation theory
  ({L}eipzig, 1994), Oper. Theory Adv. Appl., vol.~95, Birkh\"{a}user, Basel,
  1997, pp.~37--68. \MR{1473250}

\bibitem[ABK02]{MR1873434}
Daniel Alpay, Vladimir Bolotnikov, and H.~Turgay Kaptano\u{g}lu, \emph{The
  {S}chur algorithm and reproducing kernel {H}ilbert spaces in the ball},
  Linear Algebra Appl. \textbf{342} (2002), 163--186. \MR{1873434}

\bibitem[AD86]{MR902603}
Daniel Alpay and Harry Dym, \emph{On applications of reproducing kernel spaces
  to the {S}chur algorithm and rational {$J$} unitary factorization}, I.
  {S}chur methods in operator theory and signal processing, Oper. Theory Adv.
  Appl., vol.~18, Birkh\"{a}user, Basel, 1986, pp.~89--159. \MR{902603}

\bibitem[AD93]{MR1200633}
\bysame, \emph{On a new class of structured reproducing kernel spaces}, J.
  Funct. Anal. \textbf{111} (1993), no.~1, 1--28. \MR{1200633}

\bibitem[AD06]{MR2223568}
D.~Alpay and C.~Dubi, \emph{Some remarks on the smoothing problem in a
  reproducing kernel {H}ilbert space}, J. Anal. Appl. \textbf{4} (2006), no.~2,
  119--132. \MR{2223568}

\bibitem[ADRdS97]{MR1465432}
Daniel Alpay, Aad Dijksma, James Rovnyak, and Hendrik de~Snoo, \emph{Schur
  functions, operator colligations, and reproducing kernel {P}ontryagin
  spaces}, Operator Theory: Advances and Applications, vol.~96, Birkh\"{a}user
  Verlag, Basel, 1997. \MR{1465432}

\bibitem[AFMP94]{MR1305949}
Gregory~T. Adams, John Froelich, Paul~J. McGuire, and Vern~I. Paulsen,
  \emph{Analytic reproducing kernels and factorization}, Indiana Univ. Math. J.
  \textbf{43} (1994), no.~3, 839--856. \MR{1305949}

\bibitem[AG08]{MR2397845}
Daniel Alpay and Israel Gohberg, \emph{Inverse problems for first-order
  discrete systems}, Recent advances in matrix and operator theory, Oper.
  Theory Adv. Appl., vol. 179, Birkh\"{a}user, Basel, 2008, pp.~1--24.
  \MR{2397845}

\bibitem[AJ15]{MR3402823}
Daniel Alpay and Palle Jorgensen, \emph{Spectral theory for {G}aussian
  processes: reproducing kernels, boundaries, and {$L^2$}-wavelet generators
  with fractional scales}, Numer. Funct. Anal. Optim. \textbf{36} (2015),
  no.~10, 1239--1285. \MR{3402823}

\bibitem[AJ21]{MR4302453}
Daniel Alpay and Palle E.~T. Jorgensen, \emph{New characterizations of
  reproducing kernel {H}ilbert spaces and applications to metric geometry},
  Opuscula Math. \textbf{41} (2021), no.~3, 283--300. \MR{4302453}

\bibitem[AJK15]{MR3447225}
Daniel Alpay, Palle E.~T. Jorgensen, and David~P. Kimsey, \emph{Moment problems
  in an infinite number of variables}, Infin. Dimens. Anal. Quantum Probab.
  Relat. Top. \textbf{18} (2015), no.~4, 1550024, 14. \MR{3447225}

\bibitem[AJL17]{MR3647183}
Daniel Alpay, Palle Jorgensen, and Izchak Lewkowicz, \emph{Characterizations of
  families of rectangular, finite impulse response, para-unitary systems}, J.
  Appl. Math. Comput. \textbf{54} (2017), no.~1-2, 395--423. \MR{3647183}

\bibitem[AL95]{MR1329822}
Daniel Alpay and Juliette Leblond, \emph{Traces of {H}ardy functions and
  reproducing kernel {H}ilbert spaces}, Arch. Math. (Basel) \textbf{64} (1995),
  no.~6, 490--499. \MR{1329822}

\bibitem[Alp91]{MR1126127}
Daniel Alpay, \emph{Some reproducing kernel spaces of continuous functions}, J.
  Math. Anal. Appl. \textbf{160} (1991), no.~2, 424--433. \MR{1126127}

\bibitem[AM19]{MR3938002}
Jim Agler and John~E. McCarthy, \emph{Non-commutative functional calculus}, J.
  Anal. Math. \textbf{137} (2019), no.~1, 211--229. \MR{3938002}

\bibitem[AM24]{MR4750930}
\bysame, \emph{The {H}ardy-{W}eyl algebra}, J. Operator Theory \textbf{91}
  (2024), no.~2, 521--544. \MR{4750930}

\bibitem[AMV09]{MR2543972}
Daniel Alpay, Andrey Melnikov, and Victor Vinnikov, \emph{Un algorithme de
  {S}chur pour les fonctions de transfert des syst\`emes surd\'{e}termin\'{e}s
  invariants dans une direction}, C. R. Math. Acad. Sci. Paris \textbf{347}
  (2009), no.~13-14, 729--733. \MR{2543972}

\bibitem[AMV12]{MR2992027}
\bysame, \emph{Schur algorithm in the class {$\mathcal{SI}$} of
  {$\mathcal{J}$}-contractive functions intertwining solutions of linear
  differential equations}, Integral Equations Operator Theory \textbf{74}
  (2012), no.~3, 313--344. \MR{2992027}

\bibitem[APV23]{MR4557801}
Daniel Alpay, Ariel Pinhas, and Victor Vinnikov, \emph{Commuting operators over
  {P}ontryagin spaces with applications to system theory}, J. Funct. Anal.
  \textbf{284} (2023), no.~10, Paper No. 109864, 66. \MR{4557801}

\bibitem[Aro48]{MR31663}
N.~Aronszajn, \emph{Reproducing and pseudo-reproducing kernels and their
  application to the partial differential equations of physics}, Harvard
  University, Graduate School of Engineering, Cambridge, MA, 1948, Studies in
  partial differential equations. Technical report 5, preliminary note.
  \MR{31663}

\bibitem[AS56]{MR80878}
N.~Aronszajn and K.~T. Smith, \emph{Functional spaces and functional
  completion}, Ann. Inst. Fourier (Grenoble) \textbf{6} (1955/56), 125--185.
  \MR{80878}

\bibitem[BCE21]{MR4260603}
Chafiq Benhida, Ra\'{u}l~E. Curto, and George~R. Exner, \emph{Conditional
  positive definiteness as a bridge between {$k$}-hyponormality and
  {$n$}-contractivity}, Linear Algebra Appl. \textbf{625} (2021), 146--170.
  \MR{4260603}

\bibitem[CI24]{MR4728479}
Ra\'{u}l~E. Curto and Maria Infusino, \emph{The realizability problem as a
  special case of the infinite-dimensional truncated moment problem}, Proc.
  Amer. Math. Soc. \textbf{152} (2024), no.~5, 2145--2155. \MR{4728479}

\bibitem[CLPL19]{MR3911880}
Wen-Sheng Chen, Jingmin Liu, Binbin Pan, and Yugao Li, \emph{Block kernel
  nonnegative matrix factorization for face recognition}, Int. J. Wavelets
  Multiresolut. Inf. Process. \textbf{17} (2019), no.~1, 1850059, 28.
  \MR{3911880}

\bibitem[CY17]{MR3700848}
Rupak Chatterjee and Ting Yu, \emph{Generalized coherent states, reproducing
  kernels, and quantum support vector machines}, Quantum Inf. Comput.
  \textbf{17} (2017), no.~15-16, 1292--1306. \MR{3700848}

\bibitem[DGG23]{MR4541766}
Joscha Diehl, Malte Gerhold, and Nicolas Gilliers, \emph{Shuffle algebras and
  non-commutative probability for pairs of faces}, SIGMA Symmetry Integrability
  Geom. Methods Appl. \textbf{19} (2023), Paper No. 006, 50. \MR{4541766}

\bibitem[DM20]{MR4114386}
Nicol\`o Drago and Valter Moretti, \emph{The notion of observable and the
  moment problem for {$\ast$}-algebras and their {GNS} representations}, Lett.
  Math. Phys. \textbf{110} (2020), no.~7, 1711--1758. \MR{4114386}

\bibitem[DMS12]{MR2891702}
Ronald~G. Douglas, Gadadhar Misra, and Jaydeb Sarkar, \emph{Contractive
  {H}ilbert modules and their dilations}, Israel J. Math. \textbf{187} (2012),
  141--165. \MR{2891702}

\bibitem[DS23]{MR4575370}
Susmita Das and Jaydeb Sarkar, \emph{Tridiagonal kernels and left-invertible
  operators with applications to {A}luthge transforms}, Rev. Mat. Iberoam.
  \textbf{39} (2023), no.~2, 397--437. \MR{4575370}

\bibitem[EL77]{MR489494}
D.~E. Evans and J.~T. Lewis, \emph{Dilations of irreversible evolutions in
  algebraic quantum theory}, Communications Dublin Inst. Advanced Studies. Ser.
  A (1977), no.~24, v+104. \MR{489494}

\bibitem[Eva80]{MR582649}
David~E. Evans, \emph{A review on semigroups of completely positive maps},
  Mathematical problems in theoretical physics ({P}roc. {I}nternat. {C}onf.
  {M}ath. {P}hys., {L}ausanne, 1979), Lecture Notes in Phys., vol. 116,
  Springer, Berlin-New York, 1980, pp.~400--406. \MR{582649}

\bibitem[FGFS24]{MR4762721}
Ahmad Farooq, Cristian~A. Galvis-Florez, and Simo S\"{a}rkk\"{a},
  \emph{Quantum-assisted {H}ilbert-space {G}aussian process regression}, Phys.
  Rev. A \textbf{109} (2024), no.~5, Paper No. 052410. \MR{4762721}

\bibitem[FLZG24]{MR4705975}
Lingxuan Feng, Shunlong Luo, Yan Zhao, and Zhihua Guo, \emph{Equioverlapping
  measurements as extensions of symmetric informationally complete positive
  operator valued measures}, Phys. Rev. A \textbf{109} (2024), no.~1, Paper No.
  012218, 10. \MR{4705975}

\bibitem[FS20]{MR4073554}
Martin Forde and Benjamin Smith, \emph{The conditional law of the
  {B}acry-{M}uzy and {R}iemann-{L}iouville log correlated {G}aussian fields and
  their {GMC}, via {G}aussian {H}ilbert and fractional {S}obolev spaces},
  Statist. Probab. Lett. \textbf{161} (2020), 108732, 11. \MR{4073554}

\bibitem[Gue22]{MR4414825}
J.~C. Guella, \emph{On {G}aussian kernels on {H}ilbert spaces and kernels on
  hyperbolic spaces}, J. Approx. Theory \textbf{279} (2022), Paper No. 105765,
  36. \MR{4414825}

\bibitem[HCW21]{MR4256782}
Shan Huang, Zeng-Bing Chen, and Shengjun Wu, \emph{Entropic uncertainty
  relations for general symmetric informationally complete positive
  operator-valued measures and mutually unbiased measurements}, Phys. Rev. A
  \textbf{103} (2021), no.~4, Paper No. 042205, 9. \MR{4256782}

\bibitem[Hol22]{MR4509531}
A.~S. Holevo, \emph{On optimization problem for positive operator-valued
  measures}, Lobachevskii J. Math. \textbf{43} (2022), no.~7, 1646--1650.
  \MR{4509531}

\bibitem[Jek20]{MR4067990}
David Jekel, \emph{Operator-valued chordal {L}oewner chains and non-commutative
  probability}, J. Funct. Anal. \textbf{278} (2020), no.~10, 108452, 100.
  \MR{4067990}

\bibitem[JST20]{MR4126821}
Palle Jorgensen, Myung-Sin Song, and James Tian, \emph{A {K}aczmarz algorithm
  for sequences of projections, infinite products, and applications to frames
  in {IFS} {$L^2$} spaces}, Adv. Oper. Theory \textbf{5} (2020), no.~3,
  1100--1131. \MR{4126821}

\bibitem[JT21]{MR4274591}
Palle Jorgensen and James Tian, \emph{Infinite-dimensional analysis---operators
  in {H}ilbert space; stochastic calculus via representations, and duality
  theory}, World Scientific Publishing Co. Pte. Ltd., Hackensack, NJ, [2021]
  \copyright 2021. \MR{4274591}

\bibitem[JT24a]{MR4751140}
\bysame, \emph{The operators of stochastic calculus}, Random Oper. Stoch. Equ.
  \textbf{32} (2024), no.~2, 185--205. \MR{4751140}

\bibitem[JT24b]{JT24-2}
Palle E.~T. Jorgensen and James Tian, \emph{Canonical data-reconstructions via
  kernels, hilbert space-valued gaussian processes, and quantum states}, arXiv:
  2405.02796 (2024).

\bibitem[JT24c]{JT24-1}
\bysame, \emph{Hilbert space valued gaussian processes, their kernels,
  factorizations, and covariance structure}, arXiv:2404.14685 (2024).

\bibitem[JT24d]{JT24-3}
\bysame, \emph{Positive operator-valued kernels and non-commutative
  probability}, arXiv: 2405.09315 (2024).

\bibitem[JYZ21]{MR4160499}
Hao Jiang, Ming Yi, and Shihua Zhang, \emph{A kernel non-negative matrix
  factorization framework for single cell clustering}, Appl. Math. Model.
  \textbf{90} (2021), 875--888. \MR{4160499}

\bibitem[Kat19]{MR3940383}
Victor Katsnelson, \emph{On the completeness of {G}aussians in a {H}ilbert
  functional space}, Complex Anal. Oper. Theory \textbf{13} (2019), no.~3,
  637--658. \MR{3940383}

\bibitem[KM01]{MR1881441}
Stanis\l~aw Kwapie\'{n} and Jan Mycielski, \emph{On the {K}aczmarz algorithm of
  approximation in infinite-dimensional spaces}, Studia Math. \textbf{148}
  (2001), no.~1, 75--86. \MR{1881441}

\bibitem[KM20]{MR4101087}
V\'{\i}t Kubelka and Bohdan Maslowski, \emph{Filtering of {G}aussian processes
  in {H}ilbert spaces}, Stoch. Dyn. \textbf{20} (2020), no.~3, 2050020, 22.
  \MR{4101087}

\bibitem[Lop23]{MR4641110}
Oleh Lopushansky, \emph{Bernstein-{J}ackson inequalities on {G}aussian
  {H}ilbert spaces}, J. Fourier Anal. Appl. \textbf{29} (2023), no.~5, Paper
  No. 58, 21. \MR{4641110}

\bibitem[LYA23]{MR4690276}
Binglin Li, Changwon Yoon, and Jeongyoun Ahn, \emph{Reproducing kernels and new
  approaches in compositional data analysis}, J. Mach. Learn. Res. \textbf{24}
  (2023), Paper No. [327], 34. \MR{4690276}

\bibitem[Pau02]{MR1976867}
Vern Paulsen, \emph{Completely bounded maps and operator algebras}, Cambridge
  Studies in Advanced Mathematics, vol.~78, Cambridge University Press,
  Cambridge, 2002. \MR{1976867}

\bibitem[Phi61]{MR133686}
R.~S. Phillips, \emph{The extension of dual subspaces invariant under an
  algebra}, Proc. {I}nternat. {S}ympos. {L}inear {S}paces ({J}erusalem, 1960),
  Jerusalem Academic Press, Jerusalem, 1961, pp.~366--398. \MR{133686}

\bibitem[PPPR20]{MR4085703}
Satish~K. Pandey, Vern~I. Paulsen, Jitendra Prakash, and Mizanur Rahaman,
  \emph{Entanglement breaking rank and the existence of {SIC} {POVM}s}, J.
  Math. Phys. \textbf{61} (2020), no.~4, 042203, 14. \MR{4085703}

\bibitem[PR16]{MR3526117}
Vern~I. Paulsen and Mrinal Raghupathi, \emph{An introduction to the theory of
  reproducing kernel {H}ilbert spaces}, Cambridge Studies in Advanced
  Mathematics, vol. 152, Cambridge University Press, Cambridge, 2016.
  \MR{3526117}

\bibitem[PTD23]{MR4581177}
J.~E. Pascoe and Ryan Tully-Doyle, \emph{Induced {S}tinespring factorization
  and the {W}ittstock support theorem}, Results Math. \textbf{78} (2023),
  no.~4, Paper No. 135, 15. \MR{4581177}

\bibitem[PTG15]{MR3487535}
Andr\'{e}s~Esteban P\'{a}ez-Torres and Fabio~A. Gonz\'{a}lez, \emph{Online
  kernel matrix factorization}, Progress in pattern recognition, image
  analysis, computer vision, and applications, Lecture Notes in Comput. Sci.,
  vol. 9423, Springer, Cham, 2015, pp.~651--658. \MR{3487535}

\bibitem[Qar24]{MR4712050}
Jamal~El Qars, \emph{Gaussian {R}\'{e}nyi-2 correlations in a nondegenerate
  three-level laser}, Quantum Inf. Process. \textbf{23} (2024), no.~3, Paper
  No. 83, 16. \MR{4712050}

\bibitem[QL17]{MR3709384}
Likuan Qin and Vadim Linetsky, \emph{Long-term factorization of affine pricing
  kernels}, Math. Financ. Econ. \textbf{11} (2017), no.~4, 479--498.
  \MR{3709384}

\bibitem[RMC21]{MR4280112}
Benjamin Robinson, Bill Moran, and Doug Cochran, \emph{Positive operator-valued
  measures and densely defined operator-valued frames}, Rocky Mountain J. Math.
  \textbf{51} (2021), no.~1, 265--272. \MR{4280112}

\bibitem[Sch52]{MR45307}
Laurent Schwartz, \emph{Th\'{e}orie des noyaux}, Proceedings of the
  {I}nternational {C}ongress of {M}athematicians, {C}ambridge, {M}ass., 1950,
  vol. 1, Amer. Math. Soc., Providence, RI, 1952, pp.~220--230. \MR{45307}

\bibitem[Sch64]{MR179587}
\bysame, \emph{Sous-espaces hilbertiens d'espaces vectoriels topologiques et
  noyaux associ\'{e}s (noyaux reproduisants)}, J. Analyse Math. \textbf{13}
  (1964), 115--256. \MR{179587}

\bibitem[Shi05]{MR2139107}
Serguei Shimorin, \emph{Commutant lifting and factorization of reproducing
  kernels}, J. Funct. Anal. \textbf{224} (2005), no.~1, 134--159. \MR{2139107}

\bibitem[SNFBK10]{MR2760647}
B\'{e}la Sz.-Nagy, Ciprian Foias, Hari Bercovici, and L\'{a}szl\'{o}
  K\'{e}rchy, \emph{Harmonic analysis of operators on {H}ilbert space},
  enlarged ed., Universitext, Springer, New York, 2010. \MR{2760647}

\bibitem[Sti55]{MR69403}
W.~Forrest Stinespring, \emph{Positive functions on {$C^*$}-algebras}, Proc.
  Amer. Math. Soc. \textbf{6} (1955), 211--216. \MR{69403}

\bibitem[Ver22]{MR4482713}
Dominic Verdon, \emph{A covariant {S}tinespring theorem}, J. Math. Phys.
  \textbf{63} (2022), no.~9, Paper No. 091705, 49. \MR{4482713}

\end{thebibliography}

\end{document}